\documentclass{article}

\usepackage{amsmath, amssymb, amsthm}
\usepackage{hyperref}
\usepackage{cleveref}
\usepackage{chngcntr}
\usepackage{thmtools}
\usepackage{thm-restate}
\usepackage{fancyhdr}
\usepackage{dsfont}
\usepackage[english]{babel}
\usepackage[utf8]{inputenc}
\usepackage{commath}
\usepackage{marvosym}
\usepackage{mathtools}

\usepackage{thmtools}
\usepackage{thm-restate}
\usepackage{soul}

\usepackage{tikz-cd}

\newtheorem{theorem}{Theorem}[section]
\newtheorem{corollary}[theorem]{Corollary}
\newtheorem{proposition}[theorem]{Proposition}

\newtheorem{lemma}[theorem]{Lemma}

\newtheorem*{theorem*}{Theorem}
\newtheorem*{corollary*}{Corollary}

\theoremstyle{definition}
\newtheorem{definition}[theorem]{Definition}

\newtheorem{example}[theorem]{Example}

\newtheorem{question}[theorem]{Question}

\DeclareMathOperator{\length}{\ell}
\DeclareMathOperator{\pdim}{pd}
\DeclareMathOperator{\tor}{Tor}
\DeclareMathOperator{\ext}{Ext}

\DeclareMathOperator{\depth}{depth}
\DeclareMathOperator{\Hom}{Hom}
\DeclareMathOperator{\supp}{Supp}

\DeclareMathOperator{\grade}{grade}

\title{A Rigidity Theorem for Ext}
\author{Andrew J. Soto Levins}
\date{}

\begin{document}
\maketitle
\begin{abstract}
The goal of this paper is to show that if $R$ is an unramified hypersurface, if $M$ and $N$ are finitely generated $R$ modules, and if $\ext_{R}^{n}(M,N)=0$ for some $n\leq\grade{M}$, then $\ext_{R}^{i}(M,N)=0$ for $i\leq n$. A corollary of this says that $\ext_{R}^{i}(M,M)\neq 0$ for $i\leq\grade{M}$ and $M\neq 0$. These results are related to a question of Jorgensen and results of Dao. 
\end{abstract}

\section{Introduction} 
Throughout this paper all rings are commutative and $(R,\mathfrak{m},k)$ is a Noetherian local ring. The goal of this paper is to prove a rigidity theorem for Ext. 

\begin{definition} Let $M$ be a nonzero finitely generated $R$ module. We say that $M$ is \emph{Ext rigid} if, for all finitely generated $R$ modules $N$, $\ext_{R}^{n}(M,N)=0$ for some $n\leq\grade{M}$ implies $\ext_{R}^{i}(M,N)=0$ for $i\leq n$. (Recall $\grade{M}=\inf\{i|\ext_{R}^{i}(M,R)\neq 0\}$.)
\end{definition}

\begin{theorem*} (Theorem \ref{maintheorem} below) Let $R$ be a regular local ring or an unramified hypersurface and let $M$ be a finitely generated $R$ module. Then $M$ is Ext rigid.
\end{theorem*}

Recall that $R$ is an \emph{unramified} local ring if $R$ is equicharacteristic or is mixed characteristic and char($k$)$\not\in\mathfrak{m}^{2}$. Also, $R$ is a \emph{hypersurface} if it is not regular and the $\mathfrak{m}$-adic completion $\hat{R}$ of $R$ is the quotient of a regular local ring by a nonzerodivisor: $\hat{R}\cong Q/x$ where $(Q,\mathfrak{n},k)$ is a regular local ring and $x\in\mathfrak{n}^{2}$ is a nonzerodivisor. We give the following as a corollary to the main theorem:

\begin{corollary*} (Corollary \ref{maincorollary} below) Let $R$ be an unramified hypersurface and let $M$ be a nonzero finitely generated $R$ module. Then $\ext_{R}^{i}(M,M)\neq 0$ for $0\leq i\leq\grade{M}$.
\end{corollary*}

For a Noetherian local ring $R$ and a finitely generated CM module $M$ of finite projective dimension, a result of Ischebek \cite[Theorem 17.1]{matsumura} gives us $\ext_{R}^{i}(M,R)=0$ for 
$$i<\depth{R}-\dim{M}=\depth{R}-\depth{M}=\pdim_{R}M.$$ 
Since $\ext_{R}^{\pdim_{R}M}(M,R)\neq 0$, we have $\grade{M}=\pdim_{R}M$, so this corollary answers a question of Jorgensen \cite[Question 2.7]{jorgensen} when the module is CM:

\begin{question}[Jorgensen] \label{jorgensenquestion} Let $M$ be a nonzero finitely generated module of finite projective dimension over a complete intersection $R$ of positive codimension. Does $\ext_{R}^{n}(M,M)=0$ for some $n\geq 1$ imply that $\pdim_{R}M<n$? 
\end{question}

A result of Jothilingam says that for a finitely generated module over a regular local ring $R$, $\ext_{R}^{n}(M,M)=0$ implies that $\pdim_{R}M<n$ (this can be found in \cite{jothilingam}). In \cite{jorgensen} Jorgensen gives an extension of this result and in \cite[Proposition 5.4]{dao2} Dao shows:

\begin{proposition}[Dao] Let $R$ be an unramified hypersurface and let $M$ be a finitely generated $R$ module. Then the answer to Jorgensen's Question \ref{jorgensenquestion} is ``yes" provided $[M]=0$ in $\overline{G}(R)_{\mathbb{Q}}$. 
\end{proposition}

Recall that $G(R)$ is the Grothendieck group of finitely generated modules over $R$ and $\overline{G}(R)_{\mathbb{Q}}=\overline{G}(R)\otimes_{\mathbb{Z}}\mathbb{Q}$ where $\overline{G}(R)=G(R)/\mathbb{Z}\cdot[R]$ is the reduced Grothendieck group. Note that in our corollary we do not need the module to be zero in $\overline{G}(R)_{\mathbb{Q}}$.\newline

Our main theorem is motivated by the following result of Dao \cite[Proposition 2.8]{dao}:

\begin{theorem}[Dao] \label{daointro} Let $R$ be an unramified hypersurface and let $M$ and $N$ be finitely generated $R$ modules so that $\length(\tor_{i}^{R}(M,N))<\infty$ for $i\gg 0$. If $\theta^{R}(M,N)=0$, then $(M,N)$ is Tor rigid.
\end{theorem}

For a finitely generated module $M$, we say that $M$ is \emph{Tor rigid} if, for all finitely generated $R$ modules $N$, $\tor_{n}^{R}(M,N)=0$ for some $n\geq 0$ implies $\tor_{i}^{R}(M,N)=0$ for all $i\geq n$. For a hypersurface $R$ and finitely generated $R$ modules $M$ and $N$ with $\length(\tor_{i}^{R}(M,N))<\infty$ for $i\gg 0$, define $\theta^{R}(M,N)$ to be
$$\theta^{R}(M,N)=\length(\tor_{2j}^{R}(M,N))-\length(\tor_{2j+1}^{R}(M,N))$$
for $j\gg 0$. Results of Auslander \cite{auslander} and of Lichtenbaum \cite{lichtenbaum} show that every module over a regular local ring is Tor rigid. Auslander originally used this fact to study torsion free modules. In \cite{hochster} Hochster defined the pairing $\theta^{R}$. Note that unlike Dao's theorem, the $\theta^{R}$ pairing is not used in our main theorem and we do not need $\length(\ext_{R}^{i}(M,N))<\infty$ for $i\gg 0$.\newline

I would like to thank my advisor Mark Walker for his continuous support and  patience. His comments enormously contributed to the flow of the proofs and the paper. I would also like to thank David Jorgensen for his comments on an earlier draft of this paper.

\section{The Main Theorem and its Corollary}
In this section we prove the main theorem and its corollary. We prove the main theorem in two steps. We first show that it is true when $\ext_{R}^{i}(M,N)$ has finite length for all $i$ and then prove the general case by reducing to the finite length case. We start by giving a definition and a lemma:

\begin{definition} Let $M$ be a finitely generated $R$ module and let $g=\grade{M}$. If
$$\dots\xrightarrow{\partial_{2}}F_{1}\xrightarrow{\partial_{1}}F_{0}\xrightarrow{\partial_{0}}0$$
is a minimal free resolution of $M$ over $R$, define $E_{R}(M)$ to be
$$E_{R}(M)=\text{coker}(F_{g-1}^{*}\xrightarrow{\partial_{g}^{*}}F_{g}^{*}),$$
where $F_{i}^{*}=\Hom_{R}(F_{i},R)$.
\end{definition}

\begin{lemma} \label{exttotor} Let $M$ be a finitely generated $R$ module and let $g=\grade{M}$. If $N$ is an $R$ module, then
$$\ext_{R}^{i}(M,N)\cong\tor_{g-i}^{R}(E_{R}(M),N)$$
for $0\leq i\leq g-1$.
\end{lemma}

\begin{proof}
Let 
$$F_{\bullet}=(\dots\rightarrow F_{1}\rightarrow F_{0}\rightarrow 0)$$
be a minimal free resolution of $M$ over $R$ and consider
$$F_{\leq g}=(0\rightarrow F_{g}\rightarrow\dots\rightarrow F_{0}\rightarrow 0).$$
Since $g=\grade{M}$, we have $\ext_{R}^{i}(M,R)=0$ for $0\leq i\leq g-1$, so
$$F_{\leq g}^{*}=(0\rightarrow F_{0}^{*}\rightarrow\dots\rightarrow F_{d}^{*}\rightarrow 0)$$
is a minimal free resolution of $E_{R}(M)$ over $R$. The result follows from the fact that $\Hom_{R}(F_{\leq g},N)$ and $F_{\leq g}^{*}\otimes_{R}N$ are isomorphic.
\end{proof}

Recall the following definition:

\begin{definition} Fix $j$. Let $Q$ be a ring and let $M$ and $N$ be $Q$ modules with $\length(\tor_{i}^{Q}(M,N))<\infty$ for $i\geq j$. If $\tor_{i}^{Q}(M,N)=0$ for $i\gg 0$, then define $\chi_{j}^{Q}(M,N)$ to be: 
$$\chi_{j}^{Q}(M,N)=\sum_{i\geq j}(-1)^{i-j}\length(\tor_{i}^{Q}(M,N)).$$
\end{definition}

The following theorem was proven by Hochster and Lichtenbaum. It was proved by Lichtenbaum in \cite{lichtenbaum} in most cases. The rest of the cases were proved by Hochster in \cite{hochster2}.

\begin{theorem}[Hochster and Lichtenbaum] \label{lichtenbaumhochster} Let $(Q,\mathfrak{n},k)$ be an unramified regular local ring and let $M$ and $N$ be finitely generated $Q$ modules. Let $j\geq0$ be an integer and  assume that $\length(\tor_{i}^{Q}(M,N))<\infty$ for $i\geq j$. Then the following hold:\newline

1. $\chi_{j}^{Q}(M,N)\geq0$.\newline

2. If $j\geq 1$, then $\chi_{j}^{Q}(M,N)=0$ if and only if $\tor_{i}^{Q}(M,N)=0$ for all $i\geq j$.
\end{theorem}

We make the following definition and prove an analog of the above theorem for Ext.

\begin{definition} Let $Q$ be a ring and let $M$ and $N$ be $Q$ modules. Let $j\geq 0$ be an integer and assume that $\length(\ext_{Q}^{i}(M,N))<\infty$ for $i\leq j$. Define $\overline{\xi}_{j}^{Q}(M,N)$ to be

$$\overline{\xi}_{j}^{Q}(M,N)=\sum_{i=0}^{j}(-1)^{-i}\length(\ext_{Q}^{j-i}(M,N)).$$
\end{definition}

\begin{lemma} \label{analogofLH} Let $(Q,\mathfrak{n},k)$ be an unramified regular local ring and let $M$ and $N$ be finitely generated $Q$ modules. Let $i$ be an integer so that $1\leq i\leq\grade{M}-1$ and assume that $\length(\ext_{Q}^{j}(M,N))<\infty$ for $j\leq i$. Then the following hold:\newline

1. $\overline{\xi}_{i}^{Q}(M,N)\geq 0$.\newline

2. $\overline{\xi}_{i}^{Q}(M,N)=0$ if and only if $\ext_{Q}^{j}(M,N)=0$ for $j\leq i$.
\end{lemma}

\begin{proof}
Let $g=\grade{M}$. Then $\ext_{Q}^{j}(M,N)=\tor_{g-j}^{Q}(E_{Q}(M),N)$ for $0\leq j\leq g-1$ by Lemma \ref{exttotor}, so by Theorem \ref{lichtenbaumhochster} due to Hochster and Lichtenbaum we have 
$$\overline{\xi}_{i}^{Q}(M,N)=\chi_{g-i}^{Q}(E_{Q}(M),N)\geq 0.$$ 
Also, if $\overline{\xi}_{i}^{Q}(M,N)=0$, then $\chi_{g-i}^{Q}(E_{Q}(M),N)=0$. Therefore 

$$\ext_{Q}^{j}(M,N)=\tor_{g-j}^{Q}(E_{Q}(M),N)=0$$
for $j\leq i$ again using Theorem \ref{lichtenbaumhochster}.
\end{proof}

\begin{lemma} \label{analogofDAO} Let $R$ be an unramified hypersurface and let $M$ and $N$ be finitely generated $R$ modules. Let $i$ be an integer so that $i\leq\grade{M}$ and assume $\length(\ext_{R}^{j}(M,N))<\infty$ for $j\leq i$. If $\ext_{R}^{i}(M,N)=0$, then $\ext_{R}^{j}(M,N)=0$ for $j\leq i$.
\end{lemma}

\begin{proof}
Note that the completion of an unramified hypersurface is an unramified hypersurface. Also, $\hat{R}$ is a faithfully flat extension of $R$, so by the Cohen structure theorem and \cite[Theorem 21.1]{matsumura} we can assume $R=Q/x$ for some unramified regular local ring $(Q,\mathfrak{n},k)$ and some nonzerodivisor $x\in \mathfrak{n}^{2}$.\newline

We can also assume that $i>0$. Suppose $\ext_{R}^{i}(M,N)=0$ and consider the following long exact sequence:

\begin{tikzcd}[cells={text width={width("$\ext_{R}^{-1}(M,N) \rightarrow\dots$")},align=center},
column sep=2em]
 0 \rightarrow \ext_{R}^{0}(M,N) \arrow{r}  & \ext_{Q}^{0}(M,N) \arrow{r} \arrow[d, phantom, ""{coordinate, name=Z}]  & \ext_{R}^{-1}(M,N) \rightarrow\dots \arrow[dll,rounded corners,to path={ -- ([xshift=2ex]\tikztostart.east)|- (Z) [near end]\tikztonodes-| ([xshift=-2ex]\tikztotarget.west)-- (\tikztotarget)}]    \\
   \ext_{R}^{i-1}(M,N) \arrow{r} & \ext_{Q}^{i-1}(M,N) \arrow{r} \arrow[d, phantom, ""{coordinate, name=Z}]  & \ext_{R}^{i-2}(M,N) \arrow[dll,rounded corners,to path={ -- ([xshift=2ex]\tikztostart.east)|- (Z) [near end]\tikztonodes-| ([xshift=-2ex]\tikztotarget.west)-- (\tikztotarget)}] \\
 \ext_{R}^{i}(M,N) \arrow{r}  & \ext_{Q}^{i}(M,N) \arrow{r}{\phi} \arrow[d, phantom, ""{coordinate, name=Z}]   & \ext_{R}^{i-1}(M,N) \arrow[dll,rounded corners,to path={ -- ([xshift=2ex]\tikztostart.east)|- (Z) [near end]\tikztonodes-| ([xshift=-2ex]\tikztotarget.west)-- (\tikztotarget)}] \\
 \text{coker}(\phi) \arrow{r} & 0.
\end{tikzcd}

It follows that 
$$\overline{\xi}_{i}^{Q}(M,N)+\length(\text{coker}(\phi))=\length(\ext_{R}^{i}(M,N)).$$
Let $g=\grade_{R}M$. Then Lemma 2 in section 18 of \cite{matsumura} implies $g=\grade_{Q}M-1$, so we have $1\leq i\leq g<\grade_{Q}M$. Therefore by Lemma \ref{analogofLH} we have $\overline{\xi}_{i}^{Q}(M,N)\geq 0$, so
$$0\leq \overline{\xi}_{i}^{Q}(M,N)+\length(\text{coker}(\phi))=\length(\ext_{R}^{i}(M,N))=0.$$
This gives us $\ext_{R}^{j}(M,N)=0$ for $j\leq i$ by Lemma \ref{analogofLH} again.
\end{proof}

We can now prove the main theorem. The proof of this theorem is similar to the proof of Proposition 2.8 in \cite{dao}.
\begin{theorem} \label{maintheorem} Let $R$ be a regular local ring or an unramified hypersurface and let $M$ be a finitely generated $R$ module. Then $M$ is Ext rigid.
\end{theorem}

\begin{proof}
Let $N$ be a finitely generated $R$ module. The theorem follows from a result of Jothilingam in the regular case: when $R$ is a regular local ring, if $\ext_{R}^{n}(M,N)=0$ for some $n\leq\grade{M}$, the Lemma in \cite{jothilingam} tells us that
$$\ext_{R}^{n-1}(M,R)\otimes_{R}N\cong\ext_{R}^{n-1}(M,N).$$
Since $n-1<\grade{M}$, we have $\ext_{R}^{n-1}(M,R)=0$, so $\ext_{R}^{n-1}(M,N)=0$.\newline

Now assume that $R$ is an unramified hypersurface. We can assume $R$ is complete as in the previous proof, and hence by \cite[Lemma 3.4]{auslander} $R_{p}$ is an unramified hypersurface for all primes $p$ of $R$. Also, since 
$$\grade{M}=\inf\{i|\ext_{R}^{i}(M,R)\neq 0\},$$
we have $\grade{M}\leq\grade_{R_{p}}{M_{p}}$ for all $p\in\supp{M}$.\newline

We proceed by induction on $\dim{R}$. Since $\grade{M}\leq\depth{R}=\dim{R}$, there is nothing to show when $\dim{R}=0$, so assume $\dim{R}>0$ and suppose $\ext_{R}^{i}(M,N)=0$ for some $i\leq\grade{M}$.\newline

Let $p$ be a nonmaximal prime in $R$. If $M_{p}=0$, then $\ext_{R_{p}}^{j}(M_{p},N_{p})=0$ for all $j$. If $M_{p}\neq 0$, then induction tells us that $\ext_{R_{p}}^{j}(M_{p},N_{p})=0$ for $j\leq i$ because $\ext_{R_{p}}^{i}(M_{p},N_{p})=0$, $i\leq\grade{M}\leq\grade_{R_{p}}{M_{p}}$, and $\dim{R_{p}}<\dim{R}$. This gives us $\ext_{R_{p}}^{j}(M_{p},N_{p})=0$ for $j\leq i$ and all nonmaximal primes $p$, and so $\length(\ext_{R}^{j}(M,N))<\infty$ for $j\leq i$. Since $\ext_{R}^{i}(M,N)=0$, Lemma \ref{analogofDAO} tells us that $\ext_{R}^{j}(M,N)=0$ for $j\leq i$.
\end{proof}

\begin{corollary} \label{maincorollary} Let $R$ be an unramified hypersurface and let $M$ be a nonzero finitely generated $R$ module. Then $\ext_{R}^{i}(M,M)\neq 0$ for $0\leq i\leq\grade{M}$. In particular, the answer to Jorgensen's Question \ref{jorgensenquestion} is ``yes" for such $R$ and all CM $R$ modules $M$.
\end{corollary}

\begin{proof}
Since $\Hom_{R}(M,M)\neq 0$ and since $\pdim_{R}M=\grade{M}$ for CM $R$ modules with finite projective dimension, this follows from Theorem \ref{maintheorem}.
\end{proof}

As noted in the proof of Corollary \ref{maincorollary}, the main theorem gives the following:
\begin{quote}
Assume $R$ is a regular local ring or an unramified hypersurface and let $M$ be a finitely generated CM $R$ module with finite projective dimension. If $N$ is a finitely generated $R$ module and if $\ext_{R}^{n}(M,N)=0$ for some $n\leq\pdim_{R}M$, then $\ext_{R}^{i}(M,N)=0$ for $i\leq n$.
\end{quote}
One might hope that this property would hold after dropping the CM assumption, but the following example shows that this is not the case.

\begin{example} Let $Q=k[[x,y,z]]$ and let $M=Q/(x^{2},xy,xz)$. Then $\pdim_{Q}M=3$, $\ext_{Q}^{2}(M,Q)=0$, and $\ext_{Q}^{1}(M,Q)\neq 0$.
\end{example}


\begin{thebibliography}{1}

\bibitem{auslander} Auslander, M. {\em Modules over Unramified Regular Local Rings}, Illinois Journal of Mathematics, vol. 5, no. 4, Dec. 1961. DOI.org (Crossref), https://doi.org/10.1215/ijm/1255631585.

\bibitem{dao} Dao, Hailong {\em Decent Intersection and Tor-Rigidity for Modules over Local Hypersurfaces}, Transactions of the American Mathematical Society, vol. 365, no. 6, Nov. 2012, pp. 2803–21. DOI.org (Crossref), https://doi.org/10.1090/S0002-9947-2012-05574-7.

\bibitem{dao2} Dao, Hailong {\em Some Observations on Local and Projective Hypersurfaces}, Mathematical Research Letters, vol. 15, no. 2, 2008, pp.207-19. DOI.org (Crossref), https://doi.org/10.4310/MRL.2008.v15.n2.a1.

\bibitem{hochster} Hochster, Melvin {\em The Dimension of an Intersection in an Ambient Hypersurface}, Proceedings of the First Midwest Algebraic Geometry Conference, May 1980.

\bibitem{hochster2} Hochster, Melvin {\em Euler Characteristics over Unramified Regular Local Rings}, Illinois Journal of Mathematics, vol. 28, no. 2, June 1984. DOI.org (Crossref), https://doi.org/10.1215/ijm/1256065276.

\bibitem{jorgensen} Jorgensen, David A. {\em Finite Projective Dimension and the Vanishing of $\ext_{R}(M,M)$}, Communications in Algebra, vol. 36, no. 12, Dec. 2008, pp.4461-71. DOI.org (Crossref), https://doi.org/10.1080/00927870802179560.

\bibitem{jothilingam} Jothilingam, P. {\em A Note on Grade}, Nagoya Mathematical Journal, vol. 59, Dec. 1975, pp.149-52. DOI.org (Crossref), https://doi.org/10.1017/S0027763000016858.

\bibitem{lichtenbaum} Lichtenbaum, Stephen {\em On the Vanishing of Tor in Regular Local Rings}, Illinois Journal of Mathematics, vol. 10, no. 2, June 1966. DOI.org (Crossref), https://doi.org/10.1215/ijm/1256055103.

\bibitem{matsumura} Matsumura, Hideyuki {\em Commutative Ring Theory}, Cambridge University Press, 1986.
	

\end{thebibliography}
\end{document}